\documentclass{article}

\usepackage{amsmath,amsfonts,amsthm,bm}

\newtheorem{theorem}{Theorem}[section]

\def\1{\bm{1}}

\def\rvu{{\mathbf{i}}}

\def\rvl{{\mathbf{l}}}

\def\rvu{{\mathbf{u}}}

\def\rmH{{\mathbf{H}}}

\def\vmu{{\bm{\mu}}}
\def\vnu{{\bm{\nu}}}
\def\vtheta{{\bm{\theta}}}
\def\vdelta{{\bm{\delta}}}

\def\vb{{\bm{b}}}
\def\vc{{\bm{c}}}
\def\vd{{\bm{d}}}
\def\ve{{\bm{e}}}
\def\vf{{\bm{f}}}

\def\vl{{\bm{l}}}

\def\vp{{\bm{p}}}
\def\vq{{\bm{q}}}

\def\vs{{\bm{s}}}

\def\vu{{\bm{u}}}
\def\vv{{\bm{v}}}
\def\vw{{\bm{w}}}
\def\vx{{\bm{x}}}
\def\vy{{\bm{y}}}
\def\vz{{\bm{z}}}

\def\mH{{\bm{H}}}

\def\mW{{\bm{W}}}

\def\mBeta{{\bm{\beta}}}

\DeclareMathAlphabet{\mathsfit}{\encodingdefault}{\sfdefault}{m}{sl}
\SetMathAlphabet{\mathsfit}{bold}{\encodingdefault}{\sfdefault}{bx}{n}

\def\gH{{\mathcal{H}}}

\def\gS{{\mathcal{S}}}

\def\gX{{\mathcal{X}}}

\def\sR{{\mathbb{R}}}

\def\ftheta{{f_{\vtheta}}}

\newcommand{\R}{\mathbb{R}}

\DeclareMathOperator*{\argmin}{arg\,min}

\newcommand{\ubar}[1]{\underline{#1}}

\newcommand{\pg}{\bm{p}} 
\newcommand{\pd}{\bm{d}} 

\newcommand{\y}{\mathbf{y}}
\newcommand{\z}{\mathbf{z}}

\newcommand{\M}{\mathbf{M}}
\newcommand{\N}{\mathbf{N}}
\newcommand{\K}{\mathbf{K}}
\newcommand{\w}{\mathbf{w}}

\newcommand{\Mth}{M^{\text{th}}}

\newcommand{\pfmax}{\bm{\bar{f}}}

\newcommand{\pgmin}{\bm{\ubar{p}}}
\newcommand{\pgmax}{\bm{\bar{p}}}

\newcommand{\xith}{\xi^{\text{th}}}

\newcommand{\muPgMin}{\ubar{\mu}}
\newcommand{\muPgMax}{\bar{\mu}}
\newcommand{\nuThMin}{\ubar{\nu}}
\newcommand{\nuThMax}{\bar{\nu}}

\usepackage{microtype}
\usepackage{booktabs} 
\usepackage{hyperref}
\usepackage{graphicx}
\usepackage[caption=false,font=footnotesize]{subfig}

\newcommand{\revision}[2][black]{\textcolor{#1}{#2}}

\usepackage[accepted]{icml2024}

\usepackage{amsmath}
\usepackage{amssymb}
\usepackage{mathtools}
\usepackage{amsthm}

\usepackage[capitalize,noabbrev]{cleveref}

\usepackage[textsize=tiny]{todonotes}

\icmltitlerunning{Compact Optimality Verification for Optimization Proxies}

\begin{document}

\twocolumn[
\icmltitle{Compact Optimality Verification for Optimization Proxies}

\begin{icmlauthorlist}
\icmlauthor{Wenbo Chen}{isye,ai4opt}
\icmlauthor{Haoruo Zhao}{isye,ai4opt}
\icmlauthor{Mathieu Tanneau}{isye,ai4opt}
\icmlauthor{Pascal Van Hentenryck}{isye,ai4opt}
\end{icmlauthorlist}

\icmlaffiliation{isye}{H. Milton Stewart School of Industrial and Systems Engineering, Georgia Institute of Technology, Atlanta, USA}
\icmlaffiliation{ai4opt}{NSF Artificial Intelligence Research Institute for Advances in Optimization (AI4OPT), USA}

\icmlcorrespondingauthor{Wenbo Chen}{wenbo.chen@gatech.edu}

\icmlkeywords{Neural network verification, optimization proxy, bilevel optimization}

\vskip 0.3in
]



\printAffiliationsAndNotice{}  

\begin{abstract}
Recent years have witnessed increasing interest in optimization proxies, i.e., machine learning models that approximate the input-output mapping of parametric optimization problems and return near-optimal feasible solutions. Following recent work by \cite{nellikkath2021physics}, this paper reconsiders the optimality verification problem for optimization proxies, i.e., the determination of the worst-case optimality gap over the instance distribution. The paper proposes a compact formulation for optimality verification and a gradient-based primal heuristic that brings \revision{substantial} computational benefits to the original formulation. The compact formulation is also more general and applies to non-convex optimization problems. The benefits of the compact formulation are demonstrated on 
large-scale DC Optimal Power Flow and knapsack problems. 
\end{abstract}

\section{Introduction}

In recent years, there has been a surge of interest in optimization proxies, i.e., differentiable programs that approximate the input/output mappings of parametric optimization problems. 
Parametric optimization problems arise in many application areas, including power systems operations, supply chain management, and manufacturing.
To be applicable in practice, these optimization proxies need to satisfy two key properties: they must return {\em feasible} solutions and they must return {\em high-quality} solutions. Feasibility has received significant attention in recent years. Constraint violations can be mitigated by incorporating them in the loss function  \citep{fioretto2020predicting,tran2021differentially,velloso2021combining,pan2020deepopf}. In some cases, the feasibility could be guaranteed by designing mask operations in the autoregression decoding \citep{bello2016neural,khalil2017learning,song2022flexible} or designing projection procedures \citep{joshi2019efficient,amos2017optnet,agrawal2019differentiable,tordesillas2023rayen,park2023compact,li2023learning,chen2023end}. For better efficiency, the projection procedure can be approximated with unrolled first-order methods \citep{chen2020rna,donti2021dc3,scieur2022curse,monga2021algorithm,sun2022alternating}.

However, quality guarantees for optimization proxies have received much less attention and have remained largely empirical. To the authors' knowledge, \cite{nellikkath2021physics} is the only work that provably provides worst-case guarantees for optimization proxies. They formalize the optimality verification problem as a bilevel optimization and reformulate it into a single level using traditional techniques from optimization theory. However, the resulting single-level formulation faces significant computational challenges as it introduces a large number of auxiliary decision variables and constraints. Moreover, it only applies to convex parametric optimization problems. 

This paper reconsiders the optimality verification problem and proposes a compact formulation that applies to non-convex problems and is more advantageous from a computational standpoint. Moreover, the paper proposes an effective primal heuristic, based on a (parallelized) \revision{Projected Gradient Attack} (PGA). When the problem is convex, PGA makes use of a conservative approximation of the value function of the optimization that leverages subgradients. 
To showcase the benefits of these modeling and algorithmic contributions, the paper considers two applications: the DC Optimal Power Flow (DC-OPF) problems and the knapsack problems. The DC-OPF illustrates how the compact optimality verification formulation and PGA can be applied to realistic industrial-sized test cases, while the knapsack problem shows their applicability to a non-convex problem. The paper also proposes novel Mixed-Integer Programming (MIP) encoding of the feasibility layers of the DC-OPF and knapsack optimization proxies. Experimental results are provided to highlight the benefits of the compact optimality verification formulation. 

The contributions of this paper are summarized as follows.
\begin{itemize}
    \item The paper proposes a {\em compact formulation for the optimality verification} of optimization proxies. The formulation provides computational benefits compared to the existing bilevel formulation and can be used to verify non-convex optimization problems.
    
    \item The paper proposes a (parallelized) \revision{Projected Gradient Attack} (PGA) that is effective in finding high-quality \revision{feasible solutions i.e., adversary points}. The PGA features a novel conservative approximation of convex value functions. 
    
    \item The paper demonstrates how to apply the compact optimality verification to two problems: large-scale DC Optimal Power Flows (DC-OPF) and knapsack problems. The paper contributes new feasibility restoration layers for the proxies and new encoding of these layers as MIP models for verification purposes. 

    \item Extensive experiments show the compact formulation, together with the primal heuristic, can effectively verify optimization proxies for large-scale DC-OPF problems and knapsack problems. The compact formulation also brings \revision{substantial} computational benefits compared to the bilevel formulation.
\end{itemize}

The rest of the paper is organized as follows.
Section \ref{sec:related_works} surveys the relevant literature. 
Section \ref{sec:preliminary} introduces background knowledge about parametric optimization, optimization proxies, and optimality verification.
Section \ref{sec:compact_formulation} presents the compact formulation for the optimality verification.
Section \ref{sec:pgd} discusses the projected gradient attack with value function approximation.
Section \ref{sec:DCOPF} presents the optimality verification on DC-OPF proxies and reports the numerical results on realistic industrial-size instances.
Section \ref{sec:Knapsack} presents the optimality verification for non-convex parametric optimization problems using the knapsack problem as a case study.
Section \ref{sec:conclusion} concludes the paper and discusses future research directions.
Section \ref{sec:impact_statement} includes the potential broader impact of the work.

\section{Related Work}
\label{sec:related_works}

Neural networks have been increasingly used in safety-critical applications such as autonomous driving \cite{bojarski2016end}, aviation \cite{julian2016policy} and power systems \cite{chen2022learning}.
However, it has become apparent that they may be highly sensitive to adversarial examples \cite{szegedy2013intriguing} i.e., a small input perturbation could cause a significant and undesired change in the output. 
To characterize the effects of such perturbations, recent verification algorithms typically employ one of two strategies: either they perform an exhaustive search of the input domain to identify a worst-case scenario \cite{venzke2020learning, nellikkath2021physics}, or they adopt a less computationally intensive method to approximate an upper bound for the worst-case violation \cite{chen2022deepsplit, xu2021fast, raghunathan2018semidefinite}. 

While optimality verification falls in the first category, it departs from the large body of the literature in two ways.
First, evaluating the objective of the optimality verification problem involves the solving of an optimization problem, which brings unique modeling and computational challenges.
Second, motivated by practical application in power systems, this paper considers larger and more complex input perturbations than the small $\ell_{p}$-norm perturbation often considered in existing work, as outlined in \cite{wang2021betacrown}.

It is also important to mention the work on dual optimization proxies \cite{qiu2023dual}. Dual optimization proxies provide dual feasible solutions at inference time for a particular instance. In contrast, optimality verification provides a worst-case guarantee for a distribution of instances. They complement each other naturally. \cite{qiu2023dual} show how to derive dual optimization proxies for the second-order cone relaxation of AC-OPF.



\section{Preliminaries}
\label{sec:preliminary}


\subsection{Parametric Optimization}
\label{sec:background:parametric_opt}

 Consider a parametric optimization problem
 of the form
    \begin{subequations}
    \label{eq:optimization}
    \begin{align}
        P(\vx): \quad \min_{\vy} \quad & c(\vx, \vy)\\
        \text{s.t.} \quad
            & g(\vx, \vy) \leq 0,
    \end{align}
    \end{subequations}
    where $\vx \in \R^p$ is the input \emph{parameter}, $\vy \in \R^n$ is the decision variable,  $c: \R^p \times \R^n \rightarrow \R$ is the cost function and $g: \R^p \times \R^n \rightarrow \R^m$ represents the constraints.
    The \emph{feasible set} is denoted by $\mathcal{Y}(\vx) \, {=} \, \{ \vy \, {\in} \, \mathbb{R}^{n} \, | \, g(\vx, \vy) \, {\leq} \, 0 \}$.
    The \emph{optimal value} of $P(\vx)$ is denoted by $\Phi(\vx)$ and $\Phi$ is referred to as the \emph{value function}. 
    The set of optimal solutions is denoted by $\mathcal{Y}^{*}(\vx) \subseteq \mathcal{Y}(\vx)$. 
    The parametric optimization problem can be viewed as a mapping from the input parameter \revision{$\vx \, {\in} \, \gX$} to an optimal decision \revision{$\vy^{*} \in \mathcal{Y}^{*}(\vx)$}.
    
    For instance, in the Optimal Power Flow (OPF) problem, the parameters are the electricity demand and generation costs. The optimization consists in finding the most cost-effective power generation that satisfies the physical and engineering constraints. In the Traveling Salesman Problem (TSP), the parameters are the locations and travel costs and the optimization consists in finding the shortest possible route that visits each location exactly once and returns to the starting point. 

    \revision{In a minimization problem, a feasible solution provides a primal bound, which is an upper bound on the optimal value.
    A dual bound is obtained from a dual-feasible solution and/or via branch-and-bound, and establishes a lower bound on the optimal value.
    Dual bounds are central to proving global optimality.
    Note that, when maximizing, a primal (resp dual) bound is a lower (resp. upper) bound on the optimal value.}

\subsection{Optimization Proxies}
\label{sec:background:opt_proxies}

    \begin{figure}[!t]
        \centering
        \includegraphics[width=0.9\columnwidth]{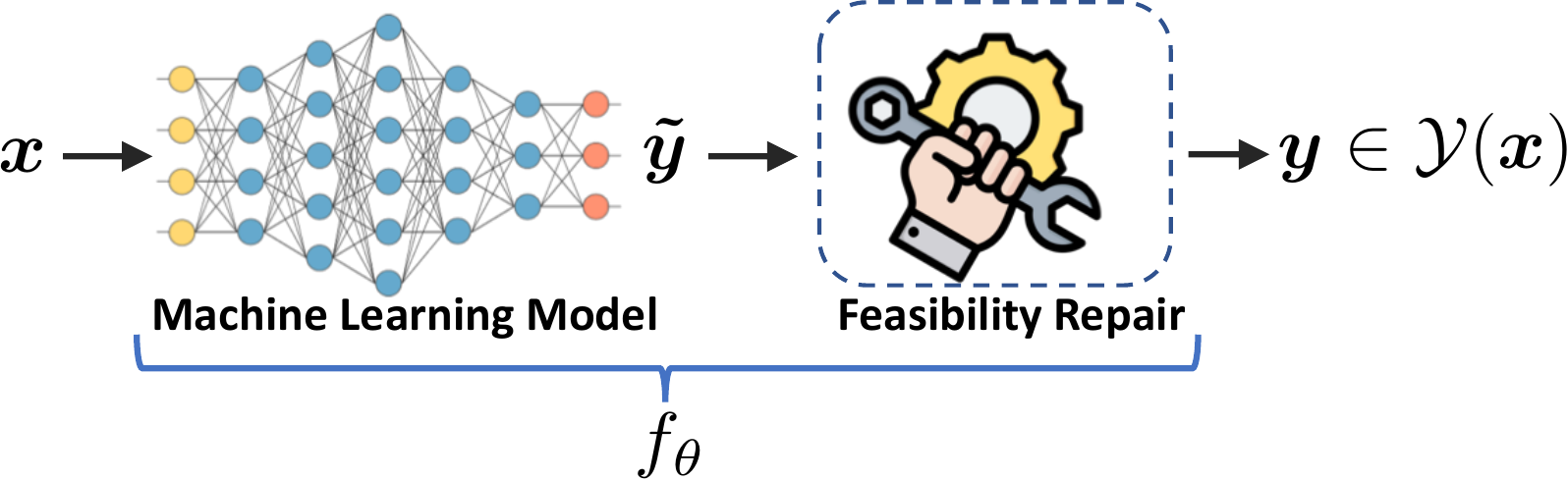}
        \caption{Optimization Proxies}
        \label{fig:optimization_proxies}
        \vspace{-1em}
    \end{figure}

    Optimization proxies (see Figure \ref{fig:optimization_proxies}) are differentiable programs 
    that approximate the mapping from input parameters $\vx$ to optimal decisions $\vy^{*}$ of the parametric optimization problem. 
    Because outputs of machine learning typically cannot satisfy complex constraints, optimization proxies consist of two parts, a machine learning model predicting the optimal decision and a feasibility repair step that projects the prediction into the feasible space. Trained optimization proxies are denoted by $f_{\theta}$, with $\theta$ denoting the weights.

    There has been significant progress in ensuring that optimization proxies produce feasible solutions e.g., \cite{li2023learning,donti2021dc3,chen2023end}. Therefore, this paper assumes that optimization proxies provide feasible solutions, i.e., 
    \begin{align}
        \forall \vx \in \gX, \ f_{\theta}(\vx) \in \mathcal{Y}(\vx).
    \end{align}

    Depending on the parametrization of the machine learning models and the modeling of the feasibility repair steps, optimization proxies could be trained using supervised learning \citep{joshi2019efficient,kotary2022fast,kotary2021learning,chen2022learning}, reinforcement learning \citep{bello2016neural,khalil2017learning,song2022flexible} and unsupervised learning \citep{karalias2020erdos,wang2022unsupervised,donti2021dc3,park2023self,chen2023end}.

\subsection{Optimality Verification}
\label{sec:background:verification}


        

    Let $\mathcal{X} \subseteq \mathbb{R}^{p}$ denote the set of possible inputs and $f_{\theta}$ be an optimization proxy. The {\em Optimality Verification Problem} can be formalized as
        \begin{align}
            \label{eq:opt_verification:value_function}
            (P_o): \quad \max_{\vx \in \gX} \quad c(\vx, f_{\theta}(\vx)) - \Phi(\vx)
        \end{align}
        Since the proxy is feasible, i.e., $f_{\theta}(\vx) \, {\in} \, \mathcal{Y}(\vx)$, it follows that $c(\vx, f_{\theta}(\vx)) \, {\geq} \, \Phi(\vx), \forall \vx \, {\in} \, \mathcal{X}$.
        In general, there is no analytical formula for $\Phi$, which is typically 
        not differentiable or even continuous. Hence $P_o$ cannot be solved directly.

        \cite{nellikkath2021physics} introduced a bilevel formulation for the optimality verification problem
        \begin{subequations}
            \label{eq:opt_verification:bilevel}
            \begin{align}
                \max_{\vx \in \gX} \quad 
                    & c(\vx, f_{\theta}(\vx)) - c(\vx, \vy^{*})\\
                \text{s.t.} \quad
                    & \vy^{*} \in \argmin_{\vy \in \mathcal{Y}(\vx)} \quad c(\vx, \vy).\label{eq:opt_verification:bilevel:lowerlevel}
            \end{align}
        \end{subequations}
        \revision{In \eqref{eq:opt_verification:bilevel}, the leader (upper level) chooses input $\vx \, {\in} \, \gX$ and computes the neural network output $f_{\theta}(\vx)$ via a MIP encoding.
        The follower (lower level) then computes an optimal solution $\vy^{*}$ of the optimization problem \eqref{eq:opt_verification:bilevel:lowerlevel}.}
        When the lower level is linear or quadratic, the bilevel formulation admits a single-level reformulation using the KKT conditions and mixed-complementarity constraints. Assuming that the proxy is a ReLU-based DNN, the overall verification problem can then be cast as an Mixed Integer Linear Programming (MILP) and solved with off-the-shelf optimization solvers like Gurobi. 
        \revision{More generally, a single-level reformulation requires strong duality assumptions and introduces additional variables and constraints, especially non-convex complementarity constraints. It creates some computational challenges as documented in the experiments. 
        Moreover, the reformulation is not available when problem $P$ is non-convex, since KKT conditions are not sufficient for optimality, and may not apply if, e.g., the lower level is discrete. In general, bilevel optimization with non-convex lower-level problems is $\Sigma_{2}^{P}$ hard \cite{caprara2013complexity}. No existing solver can solve such problems efficiently.}

\section{Compact Optimality Verification}
\label{compact_optimality_verification}

\label{sec:compact_formulation}

The core contribution of this paper is a compact formulation for the optimality verification problem, that addresses the shortcomings of the bilevel approach.
Namely, the paper proposes to formulate the optimality verification as
\begin{subequations}
\label{eq:optimality_verification_compact}
\begin{align}
    (P_c) \quad \max_{\vx \in \mathcal{X}, \vy} \quad 
    & c(\vx, \hat{\vy}) - c(\vx, \vy)\\
    s.t. \quad
    & \hat{\vy} = \ftheta(\vx),\\
    & \label{eq:optimality_verification_compact:high_point}
     \vy \in \mathcal{Y}(\vx).
\end{align}
\end{subequations}

The main difference between \eqref{eq:optimality_verification_compact} and \eqref{eq:opt_verification:bilevel} is that the inner problem \eqref{eq:opt_verification:bilevel:lowerlevel} is replaced with the simpler constraint \eqref{eq:optimality_verification_compact:high_point}.
In fact, the proposed formulation \eqref{eq:optimality_verification_compact} is the so-called high-point relaxation \cite{Moore1990_MixedIntegerLinearBilevel} of the bilevel formulation \eqref{eq:opt_verification:bilevel}.
Theorem \ref{thm:HPR_exact} shows that \eqref{eq:optimality_verification_compact} has the same optimum as \eqref{eq:opt_verification:bilevel}, i.e., the high-point relaxation is exact.
To the authors' knowledge, this result is novel.

\begin{theorem}
    \label{thm:HPR_exact}
    Problems \eqref{eq:opt_verification:bilevel} and \eqref{eq:optimality_verification_compact} have same optimum.
\end{theorem}
\begin{proof}
Recall that the compact formulation \eqref{eq:optimality_verification_compact} is a relaxation of \eqref{eq:opt_verification:bilevel}.
Therefore, it suffices to show that an optimal solution to \eqref{eq:optimality_verification_compact}, denoted by $(\tilde{\vx}, \tilde{\vy})$, is feasible for \eqref{eq:opt_verification:bilevel}.

By definition, $\tilde{\vy} \in \mathcal{Y}(\tilde{\vx})$, i.e., $\tilde{\vy}$ is feasible for the lower-level problem \eqref{eq:opt_verification:bilevel:lowerlevel} with upper-level decision $\tilde{\vx}$.
Next, assume $\tilde{\vy}$ is not optimal for \eqref{eq:opt_verification:bilevel:lowerlevel}, i.e., there exists $\hat{\vy} \in \mathcal{Y}(\tilde{\vx})$ such that \revision{$c(\tilde{\vx}, \hat{\vy}) < c(\tilde{\vx}, \tilde{\vy})$}.
By construction, $(\tilde{\vx}, \hat{\vy})$ is feasible for \eqref{eq:optimality_verification_compact} with objective value strictly better than $(\tilde{\vx}, \tilde{\vy})$, which contradicts the optimality of $(\tilde{\vx}, \tilde{\vy})$.
\end{proof}

The proposed compact formulation \eqref{eq:optimality_verification_compact} has several advantages.
First, it naturally supports non-convex constraints and objectives,
In contrast, the standard approach of reformulating bilevel problems into a single-level problem with complementarity constraints, is not possible when the lower-level problem \eqref{eq:opt_verification:bilevel:lowerlevel} is non-convex.
\emph{This is the first tractable exact formulation for verifying the optimality of non-convex optimization proxies.}
Second, even when $P$ is convex and a single-level reformulation is possible, the compact formulation avoids the additional variables and constraints that come with a single-level reformulation.
In particular, it eliminates the need for complementarity constraints, which are notoriously difficult to solve.
This last point is demonstrated in the experiments of Section \ref{sec:DCOPF}.

\section{Projected Gradient Attack}
\label{sec:pgd}

\begin{algorithm}[!t]
   \caption{Projected Gradient Attack with Value Function Approximation (PGA-VFA)}
   \label{alg:pgd}
\begin{algorithmic}
    \STATE {\bfseries Input:} Input space $\mathcal{X}$, initial point $\vx_0$, value function approximation $\widehat{\Phi}(\cdot)$, number of iterations $T$
        \FOR{$t=0$ {\bfseries to} $T$}
            \STATE $\hat{\vx}_{t+1} = \vx_{t} + \lambda \cdot \nabla_\vx \left[ c(\vx_{t}, \ftheta(\vx_{t})) - \widehat{\Phi}(\vx_{t}) \right]$ \\
            \STATE $\vx_{t+1} = \text{Proj}_{\mathcal{X}}(\hat{\vx}_{t})$ 
        \ENDFOR
\end{algorithmic}
\end{algorithm}

The projected gradient attack is highly effective in finding high-quality feasible solutions, i.e., adversarial examples, for verification problems. It can be formalized as
\begin{subequations}
\label{eq:optimality_verification_pgd}
\begin{align}
    & \hat{\vx}_{t+1} = \vx_{t} + \lambda \nabla_\vx \left[ c(\vx_{t}, \ftheta(\vx_{t})) {-} \Phi(\vx_{t}) \right], \\
    & \vx_{t+1} = \text{Proj}_{\mathcal{X}}(\hat{\vx}_{t+1}). \label{eq:optimality_verification_pgd:projection}
\end{align}
\end{subequations}
When $\mathcal{X}$ is an $\ell_{p}$ ball, the projection step in (\ref{eq:optimality_verification_pgd:projection}) can be computed in close form.

A challenge in implementing the projected gradient attack is the computation of the gradient $\nabla_\vx \Phi(\vx)$. One possible approach is to use the implicit function theorem on the KKT conditions \cite{amos2017optnet, agrawal2019differentiable}. However, this gradient computation involves the solving of many optimization instances, which may be computationally intensive for large-scale problems as those studied in this paper. {\em A second key contribution of the paper is the use in the projected gradient attack of a piece-wise linear conservative approximation of convex value function.} It builds on the following well-known result.

\begin{theorem}\cite{boyd2004convex} \label{thm:value_function_approximation}
    Consider a convex parametric optimization $P$ where (i) the objective $c$ and left-hand side $g$ do not depend on $\vx$ and (ii) the constraints' right-hand side is an affine function of $\vx$.
    Then, the value function $\Phi$ of $P$ is convex.
\end{theorem}

\begin{figure}[!t]
\centering
\includegraphics[width=0.95\columnwidth]{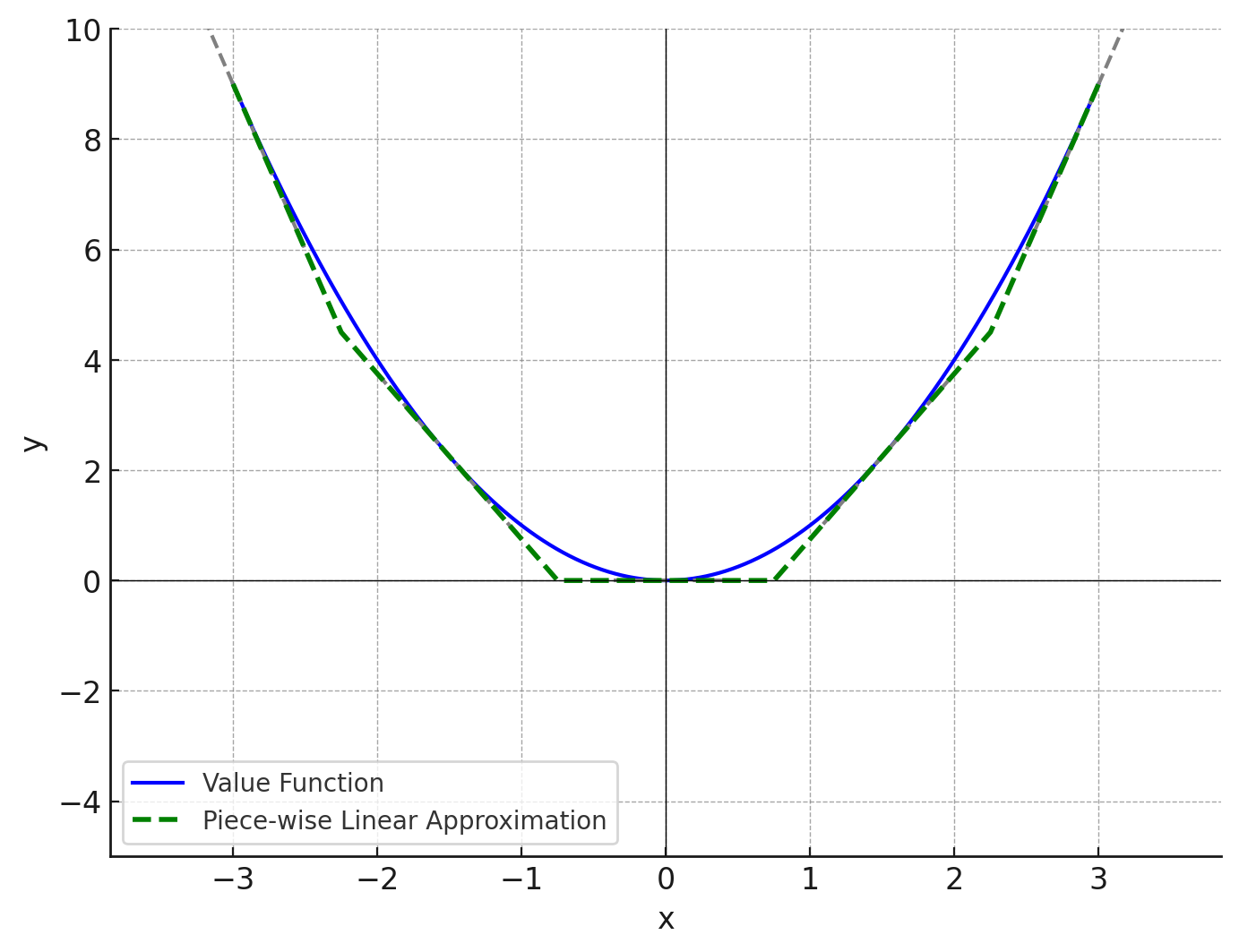}
\caption{Value Function and its Piece-wise Linear Approximation}
\label{fig:value_function_approximation}
\end{figure}

Theorem \ref{thm:value_function_approximation} implies that a convex value function can be outer-approximated by a closed-formed piecewise linear function:
\begin{align}
    \label{eq:gpa:value_func_approximation}
    \widehat{\Phi}(\vx) = \max_{i \in [N]} \Phi(\vx_i) + \nabla_{\vx} \Phi(\vx_i) ^{\intercal} (\vx - \vx_i) ,
\end{align}
where $\{\vx_i\}_{i=1}^N$ denotes a set of parameters of historical parametric optimization instances.
Moreover, the optimal dual $\boldsymbol{\lambda}_i$ of an instance with parameter $\vx_i$ represents a subgradient of the value function with respect to $\vx_i$.
If the optimal dual is unique, then $\nabla_\vx \Phi(\vx_i) = \boldsymbol{\lambda}_i$ \cite{nocedal1999numerical}.
Therefore, the approximate value function reads:
\begin{align}
    \label{eq:gpa:value_func_approximation:with_duals}
    \widehat{\Phi}(\vx) = \max_{i \in [N]} \Phi(\vx_i) + \boldsymbol{\lambda}_i^{\intercal} (\vx - \vx_i) ,
\end{align}
The piece-wise linear approximation on the convex value function is illustrated in Figure \ref{fig:value_function_approximation}. Other models such as Input Convex Neural Networks \cite{amos2017input} could also be used for the value function approximation.

The projected gradient attack may converge to a low-quality solution due to becoming trapped in local optima. This work proposed several techniques aimed at enhancing the search process. One acceleration technique consists in sampling a set of starting points in input domain $\{\vx_0^m\}_{m=1}^M$ and run PGA-VFA($\mathcal{X}$, $\vx_0^m$) for every starting point. Another acceleration technique to improve the search is to partition the input domain into subregions, and then run Algorithm (\ref{alg:pgd}) on different subregions $\{\mathcal{X}_p\}_{p=1}^P$ in parallel i.e., running PGA-VFA($\mathcal{X}_p$, $\vx_0$) for every subregion.

\section{DC Optimal Power Flow}
\label{sec:DCOPF}
\subsection{DC-OPF Formulation}
DC Optimal Power Flow (DC-OPF) is a fundamental problem for modern power system operations. It aims at determining the least-cost generator setpoints that meet grid demands while satisfying physical and operational constraints.
With the penetration of renewable energy and distributed energy resources, the system operators must continuously monitor risk in real-time, i.e., they must quickly assess the system’s behavior under various changes in load and renewables by solving a large volume of DC-OPF problems. However, traditional optimization solvers may not be capable of solving them quickly enough for large-scale power networks \cite{chen2023real}.
Recent advancements in learning-based methods have accelerated the process of finding feasible and empirically near-optimal solutions considerably faster than conventional approaches \cite{chen2023end,zhao2022ensuring,li2023learning}. {\em This paper aims at providing formal quality guarantees for optimization proxies in this space to complement existing primal learning methods}.

Consider the DC-OPF formulation
\begin{subequations}
\label{eq:DCOPF}
\begin{align}
    \min_{\pg, \xith} \quad & \bm{c}^{\top} \pg + \Mth \bm{e}^{\top}\xith\\
    \text{s.t.} \quad
        & \bm{e}^{\top} \pg = \bm{e}^{\top} \pd,
            \label{eq:DCOPF:power_balance}\\
        & +\mH \pg + \xith  \geq -\pfmax + \mH \pd,
            \label{eq:DCOPF:PTDF:min}\\
        & - \mH \pg + \xith \geq - \pfmax - \mH \pd,
            \label{eq:DCOPF:PTDF:max}\\
        & \pgmin \leq \pg \leq \pgmax,
            \label{eq:DCOPF:dispatch_bounds}\\
        & \vp \in \sR^{B}, \xith \in \sR_{+}^{E} \label{eq:DCOPF:variables}
\end{align}
\end{subequations}                       
where $B, E$ denote the total number of buses and transmission lines in the power grid, respectively, 
$\vd \in \sR^{B}$ denotes the electricity load, $\pg$ denotes the decision variables capturing energy dispatches, $\xith$ denotes the thermal limit violations, $\pfmax$ corresponds to the flow limits on the transmission lines and $\mH \in \sR^{E\times B}$ denotes the power transfer distribution factors.
The price of electricity generation is represented by $\bm{c}$, and $\Mth$ is the price of violating thermal constraints. 
The vector $\bm{e}$ consists of all ones.
Constraint \eqref{eq:DCOPF:power_balance} ensures the global power balance i.e., the total load equals the total supply.
Constraint \eqref{eq:DCOPF:PTDF:min} and \eqref{eq:DCOPF:PTDF:max} measure the thermal violations.
Constraint \eqref{eq:DCOPF:dispatch_bounds} ensures that the outputs of the generators remain within their physical limits.

\subsection{Compact Optimality Verification}

\begin{figure}[!t]
\centering
\includegraphics[width=0.95\columnwidth]{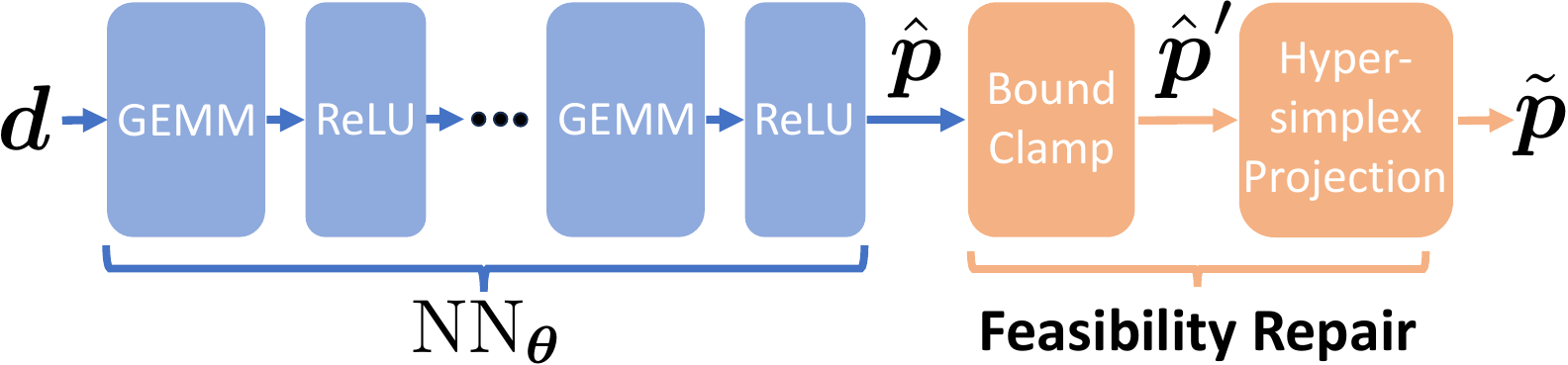}
\caption{Optimization Proxies for DCOPF}
\label{fig:optimization_proxies:dcopf}
\end{figure}
This work considers the optimization proxy proposed in \cite{chen2023end} and illustrated in Figure \ref{fig:optimization_proxies:dcopf}.
\revision{It consists of a fully-connected neural network with ReLU activation to predict the optimal dispatches and feasibility layers to ensure that the outputs satisfy the hard constraints. 
The feasibility layer consists of two parts, a bound-clamping layer, which employs a hard-sigmoid function to ensure that dispatch decisions remain within the generators' physical upper and lower generation limits \eqref{eq:DCOPF:dispatch_bounds} and a hypersimplex layer, which uses a differentiable binary search to guarantee that total power generation matches the total electricity demand, addressing constraint  \eqref{eq:DCOPF:power_balance}.}
Constraints \eqref{eq:DCOPF:PTDF:min} and \eqref{eq:DCOPF:PTDF:max} are soft and are penalized in the loss function when training the ML models. 
{\em Another contribution of this paper is a compact encoding of these layers into a mixed-integer programming formulation.} 
The detailed modeling of the compact formulation is deferred in Appendix \ref{appendix:dcopf:compact}.

\subsection{Empirical Evaluation}
\subsubsection{Experiment Setup}
The optimality verification for DCOPF proxies is evaluated over IEEE 57-/118-/300-bus and Pegase 1354-bus test cases from the PGLib library \cite{babaeinejadsarookolaee2019power}.
The data generation follows \cite{chen2023end}.
Denote by $\vd^{\text{ref}}$ the nodal load vector from the reference PGLib case.
The instances are generated by perturbing the reference load vector.
Namely, for instance $i$, $\vd^{(i)} = (\gamma^{(i)}+ \boldsymbol{\eta^{(i)}})\times\vd^{\text{ref}}$, 
where $\gamma^{(i)} \in \sR$ is a global scaling factor and $\boldsymbol{\eta^{(i)}}\in \mathbb{R}^{B}$ denotes element-wise noises.
The $\gamma$ is sampled from uniform distribution $U[80\%, 120\%]$ and for each load, $\eta$ is sampled from a uniform distribution $U[-5\%, 5\%]$.
This distribution captures system-wide correlations ($\gamma$), while allowing for local variability ($\boldsymbol{\eta}$).
The optimization proxies are trained using the self-supervised learning algorithm in \cite{chen2023end}.
It is important to note that the verification problem only depends on the weights of the trained proxy.

The parameter input domain $\mathcal{X}$ reflects the support of the distribution of instances described above.
Namely,
\begin{align*}
    &\gX = \{(\alpha {+} \mBeta)\cdot \vd^{\text{ref}}| {-}u \leq \alpha {-}1 \leq u, \mathbf{-5\%} {\leq} \mBeta {\leq} \mathbf{5\%}\},
\end{align*}
where $\alpha \, {\in} \, \sR$, $\mBeta \, {\in} \, \sR^{B}$ capture the distribution of $\gamma, \boldsymbol{\eta}$, and $u \, {\in} \, \{0, 1\%, 2\%, 5\%, 10\%, 20\%\}$ controls the size of the input domain.
Each value of $u$ yields a different optimality verification instance; note that larger values of $u$ make the instances harder to verify.

The value function approximation $\hat{\Phi}$, used in PGA-VFA (Algorithm \ref{alg:pgd}), is constructed using primal and dual solutions of 50,000 instances, generated using the above distribution.
PGA-VFA is executed in parallel, across 200 threads, using the acceleration techniques of Section \ref{sec:pgd}.
The initial step size is $10^{-3}$, and is reduced by a factor 10 if no improvement is recorded over 10 iterations.
Finally, PGA-VFA is stopped if no improved solution is found after 20 consecutive iterations, or a maximum of 500 iterations is reached.


All verification problems are solved with Gurobi 10.0 \cite{gurobi} using 16 threads and a 6-hour time limit.
Preliminary experiments revealed that Gurobi struggles to find primal-feasible solutions.
Therefore, unless specified otherwise, the $\vd^{\text{ref}} {\in} \mathcal{X}$ is always passed as a warm-start to the solver.
In addition, this work uses optimization-based bound tightening \cite{caprara2010global} to improve the MILP relaxation, by tightening the input domain of ReLU neurons; see Appendix \ref{appendix:dcopf:obbt}.
Finally, each verification instance is solved using 5 different seeds, and results are averaged using the shifted geometric mean
\begin{align*}
    \mu_{s}(x_1, \cdots, x_n) = \sqrt[n]{\Pi_{i} (x_{i} + s)} - s.
\end{align*}
The paper uses a shift $s$ of 1\% for optimality gaps and 1 for other values.
\revision{Solving times are reported in wall clock.}
Experiments are conducted on dual Intel Xeon 6226@2.7GHz machines running Linux on the cluster.


\subsubsection{Numerical Results}

    \begin{table}[!t]
        \centering
        \caption{%
            Comparison of presolved model size for Bilevel and Compact (proposed) formulations.
            \revision{Statistics are averages across 30 instances (6 distinct values of $u$ and 5 unique seeds).}
        }
        \label{tab:exp:DCOPF:model_sizes}
        {\fontsize{12pt}{15pt}\selectfont
        \resizebox{\columnwidth}{!}{
        \begin{tabular}{lrrrrrr}
            \toprule
                & \multicolumn{2}{c}{\#ConVars}
                & \multicolumn{2}{c}{\#BinVars} 
                & \multicolumn{2}{c}{\#Constraints} \\
            \cmidrule(lr){2-3}
            \cmidrule(lr){4-5}
            \cmidrule(lr){6-7}
            System &  Bilevel & Compact &  Bilevel & Compact & Bilevel & Compact\\
            \midrule
            57 & \textbf{255} & 259 & 69 & \textbf{64} & 350 & \textbf{342} \\
            118 & 609 & \textbf{522} & 218 & \textbf{100} & 862 & \textbf{619} \\
            300 & 1798 & \textbf{1320} & 861 & \textbf{325} & 2982 & \textbf{1821} \\
            1354 & 6739 & \textbf{4655} & 7281 & \textbf{1353} & 15373 & \textbf{6777} \\
             \bottomrule
         \end{tabular}}}
    \end{table}

    \begin{table}[!t]
        \centering
        \caption{%
            Comparison of solution times(s) for Bilevel and Compact (proposed) formulations across closed (i.e., solved) instances.
        }
        \label{tab:exp:DCOPF:exact_verification}
        \footnotesize
        \aboverulesep = 0.3mm \belowrulesep = 0.3mm
        \begin{tabular}{lrrrr}
            \toprule[0.15ex]
            System & \%$u$ &  Bilevel & Compact & Speedup$^{*}$  \\
            \midrule[0.1ex]
            57 
             & 0 & 0.4 & \textbf{0.3} & 22.9\% \\
             & 1 & \textbf{0.6} & 0.6 & -3.4\% \\
             & 2 & 0.6 & \textbf{0.5} & 10.4\% \\
             & 5 & 1.9 & \textbf{1.1} & 75.8\% \\
             & 10 & \textbf{1.2} & 1.2 & -1.7\% \\
             & 20 & 0.7 & \textbf{0.5} & 27.6\% \\
            \midrule[0.1ex]
            118 
             & 0 & 1.7 & \textbf{1.2} & 43.3\% \\
             & 1 & 3.9 & \textbf{2.1} & 89.7\% \\
             & 2 & 4.7 & \textbf{3.1} & 53.6\% \\
             & 5 & 28.4 & \textbf{12.1} & 135.2\% \\
             & 10 & 18.7 & \textbf{14.1} & 32.4\% \\
             & 20 & 76.8 & \textbf{59.4} & 29.3\% \\
            \midrule[0.1ex]
            300 
             & 0 & 119.8 & \textbf{63.4} & 88.9\% \\
             & 1 & 298.6 & \textbf{166.0} & 79.9\% \\
             & 2 & 1302.7 & \textbf{420.8} & 209.5\% \\
             & 5 & 17585.4 & \textbf{10507.9} & 67.4\% \\
             \bottomrule[0.15ex]
         \end{tabular}\\
         $^{*}$Average speedup compared to bilevel formulation.
    \end{table}

\begin{table}[t]
    \centering
    \caption{Primal and dual objective values attained by the Bilevel and Compact formulation on open (i.e., unsolved) instances.
    }
    \label{tab:exp:DCOPF:exact_verification_hard}
    {
    \resizebox{\columnwidth}{!}{
    \begin{tabular}{lrrrrrrr}
        \toprule
            &
            & \multicolumn{3}{c}{Primal bound (K\$)}
            & \multicolumn{3}{c}{Dual bound (M\$)} \\
        \cmidrule(lr){3-5}
        \cmidrule(lr){6-8}
        System &  \%$u$ 
            & Bilevel & Compact & \%Gain$^{\dagger}$ 
            & Bilevel & Compact & \%Gain$^{\dagger}$ \\
        \midrule
        300 
         & 10 & 466.76 & 596.98 & 27.90 & 4.45 & 3.82 & 14.08 \\
         & 20 & 410.19 & 582.81 & 42.08 & 10.79 & 9.29 & 13.88 \\
        \midrule
        1,354 
         & 0 & 94.22 & 360.03 & 282.13 & 29.64 & 26.00 & 12.25 \\
         & 1 & 84.44 & 365.50 & 332.83 & 36.35 & 35.39 & 2.63 \\
         & 2 & 50.12 & 369.68 & 637.65 & 47.64 & 45.13 & 5.26 \\
         & 5 & 54.32 & 371.52 & 584.02 & 78.79 & 74.31 & 5.69 \\
         & 10 & 46.63 & 353.20 & 657.42 & 135.32 & 129.99 & 3.94 \\
         & 20 & 49.52 & 359.72 & 626.40 & 249.02 & 235.79 & 5.31 \\
         \bottomrule
     \end{tabular}}}\\
     \footnotesize{$^{\dagger}$Relative improvement compared to bilevel formulation, in \%.}
\end{table}

\begin{figure*}
\centering
\includegraphics[width=0.33\columnwidth]{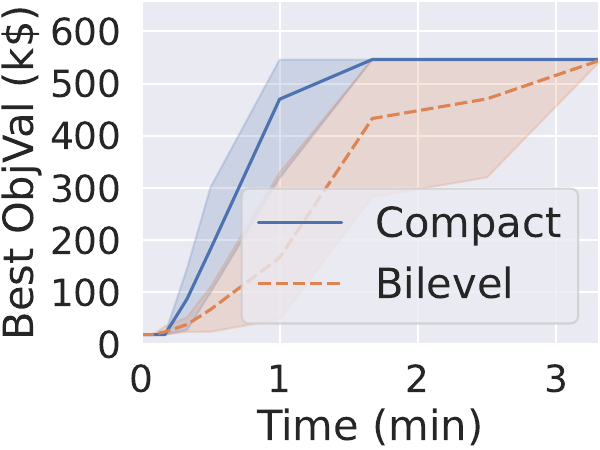}
\includegraphics[width=0.32\columnwidth]{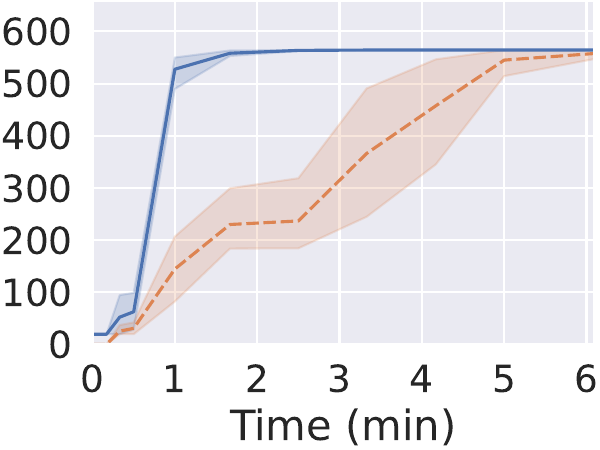}
\includegraphics[width=0.32\columnwidth]{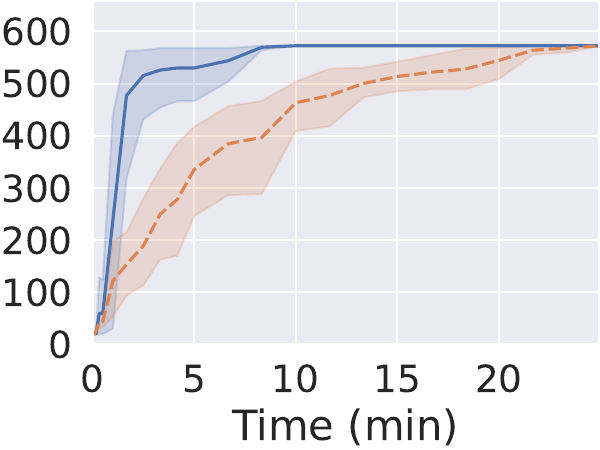}
\includegraphics[width=0.32\columnwidth]{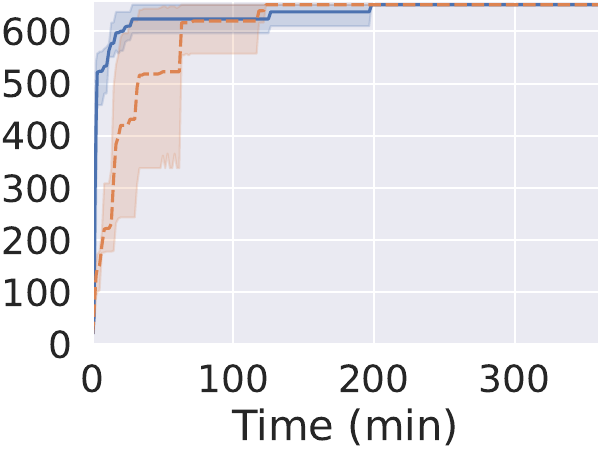}
\includegraphics[width=0.32\columnwidth]{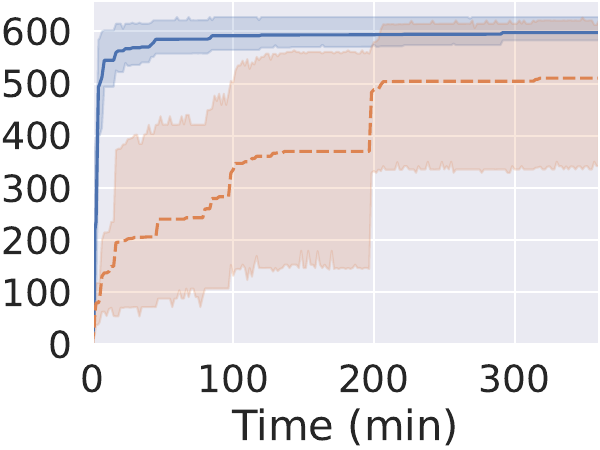}
\includegraphics[width=0.32\columnwidth]{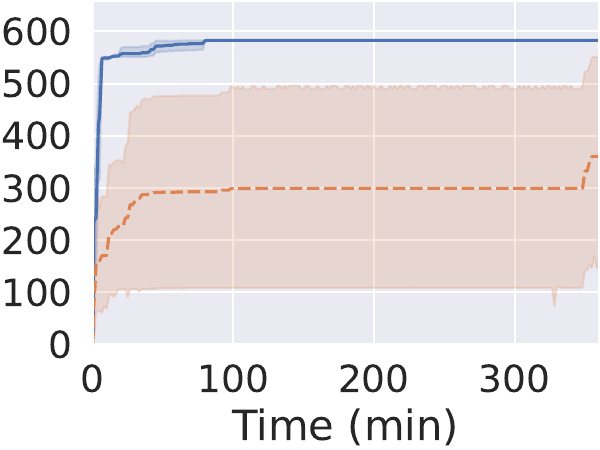} \\
\subfloat[$u=0\%$]{
\includegraphics[width=0.33\columnwidth]{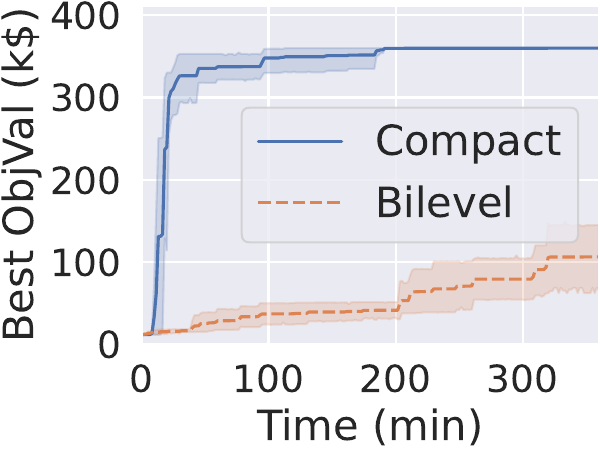}}
\subfloat[$u=1\%$]{
\includegraphics[width=0.32\columnwidth]{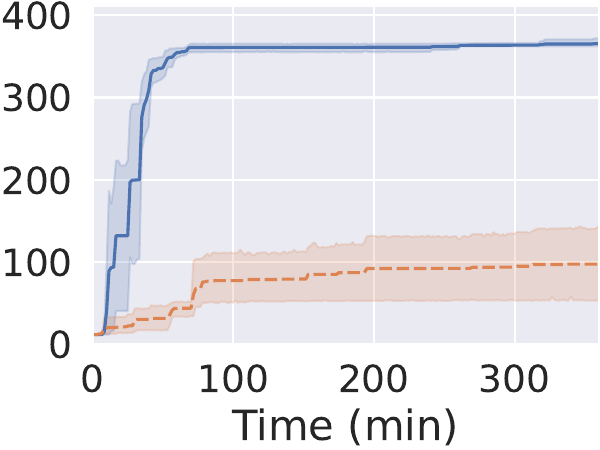}}
\subfloat[$u=2\%$]{
\includegraphics[width=0.32\columnwidth]{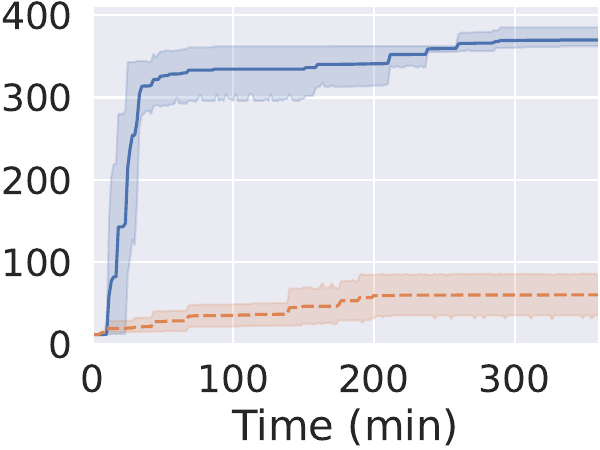}}
\subfloat[$u=5\%$]{
\includegraphics[width=0.32\columnwidth]{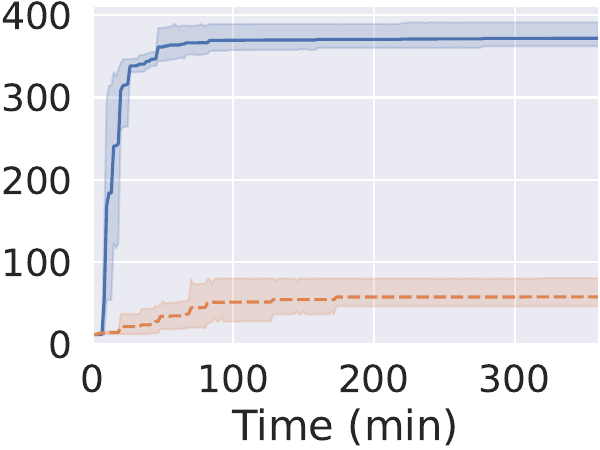}}
\subfloat[$u=10\%$]{
\includegraphics[width=0.32\columnwidth]{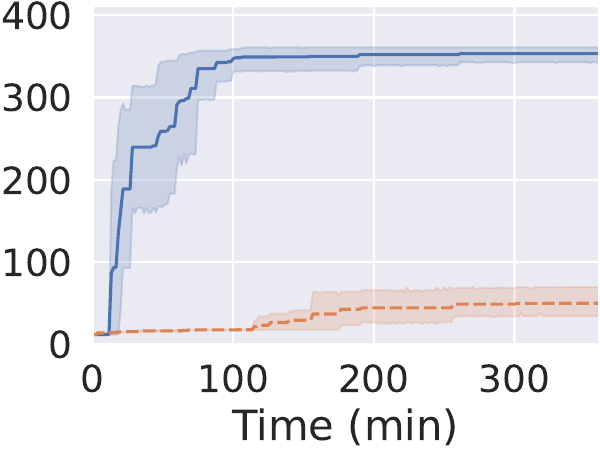}}
\subfloat[$u=20\%$]{
\includegraphics[width=0.32\columnwidth]{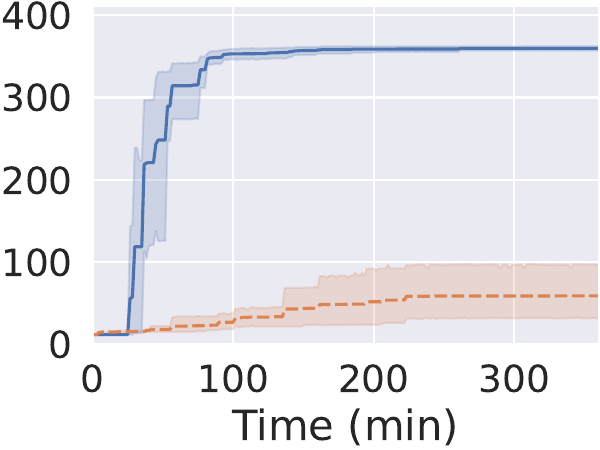}}

    \caption{Evolution of primal objective value for Compact and Bilevel formulations over time on the 300 IEEE system (top row) and the 1354 Pegase system (bottom row) with different sizes of input domains. Higher values indicate better outcomes.}
    \label{fig:primal_obj_val_along_time:DCOPF}
\end{figure*}

\begin{table}[t]
    \centering
    \caption{Comparisons of PGA-VFA and the compact formulation on final objective and solving times for large power systems.}
    \label{tab:exp:DCOPF:pga_vs_compact}
    {
    \resizebox{\columnwidth}{!}{
    \begin{tabular}{lrrrrrrr}
        \toprule
            &
            & \multicolumn{2}{c}{PGA-VFA}
            & \multicolumn{2}{c}{Compact} \\
        \cmidrule(lr){3-4}
        \cmidrule(lr){5-6}
        System &  \%$u$ &  Obj. (\$) & Time (s) &  Obj. (\$) & Time (s) & $^{\dagger}$T2PGA (s) \\
        \midrule
        300 
         & 0 & 544923.75 & 4.53 & 546038.13 & 63.41 & 52.06 \\
         & 1 & 562348.33 & 7.09 & 564263.39 & 166.01 & 114.79 \\
         & 2 & 572166.01 & 37.30 & 572464.17 & 420.83 & 307.40  \\
         & 5 & 572166.01 & 49.26 & 651163.59 & 10507.88 & 872.25 \\
         & 10 & 572166.01 & 41.41 & 596983.07 & t.o. & 2646.21  \\
         & 20 & 572166.01 & 17.86 & 582808.23 & t.o. & 2646.28 \\
        \midrule
        1,354 
         & 0 & 380673.60 & 67.26 & 360031.97 & t.o. & t.o.  \\
         & 1 & 399441.04 & 63.54 & 365497.51 & t.o. & t.o. \\
         & 2 & 403725.74 & 41.98 & 369678.23 & t.o. & t.o.\\
         & 5 & 403725.74 & 31.88 & 371524.52 & t.o. & t.o. \\
         & 10 & 571252.51 & 85.88 & 353196.16 & t.o. & t.o. \\
         & 20 & 571252.51 & 76.43 & 359722.30 & t.o. & t.o.  \\
         \bottomrule
    \end{tabular}}}\\
    \footnotesize{$^{\dagger}$Time taken by the compact formulation to reach the same primal objective as PGA; ``t.o." indicate timeouts.}
\end{table}

\begin{table}[t]
    \centering
    \caption{Comparisons of PGA-VFA and the compact formulation for different warm starts on final objective and solving times for large power systems}
    \label{tab:pga_impact}
    {
    \resizebox{\columnwidth}{!}{
    \begin{tabular}{lrrrr}
        \toprule
        System & \%$u$ & PGA-VFA & Compact+Nom & Compact+PGA\\
        \midrule
        300  
         & {\color{black}0} & 544923.75 & 546038.13 & 546038.13 \\
         & {\color{black}1} & 562348.32 & 564263.39 & 564263.44 \\
         & {\color{black}2} & 572166.01 & 572464.17 & 572464.17 \\
         & {\color{black}5} & 572166.01 & 651163.59 & 651163.38 \\
         & 10 & 572166.01 & 596983.07 & 596983.90 \\
         & 20 & 572166.01 & 582808.23 & 582808.47 \\
        \midrule
        1,354 
         & 0 & 380673.60 & 360031.97 & 380757.88 \\
         & 1 & 399441.04 & 365497.51 & 399579.63 \\
         & 2 & 403725.74 & 369678.23 & 406165.38 \\
         & 5 & 403725.74 & 371524.52 & 406164.74 \\
         & 10 & 571252.51 & 353196.16 & 571715.99 \\
         & 20 & 571252.51 & 359722.30 & 578732.15 \\
         \bottomrule
    \end{tabular}}
    }
\end{table}

\paragraph{Effectiveness of Compact Formulation}
Table \ref{tab:exp:DCOPF:model_sizes} reports the size (number of variables and constraints) of the compact formulation and bilevel formulation, after being presolved by Gurobi.
Table \ref{tab:exp:DCOPF:model_sizes} shows that the compact formulation results in substantially fewer binary decision variables than its bilevel counterpart, especially on large systems.

The size reduction of the compact formulation is reflected in the solving process.
For ease of comparison, the verification instances are split into two groups.
On the one hand, \emph{closed} instances are instances that can be solved by at least one approach within the prescribed time limit.
They include all 57-bus and 118-bus instances, as well as the 300-bus instances with $u=0, 1\%, 2\%, 5\%$.
On the other hand, \emph{open} instances are those that cannot be solved by either formulation.
They include the 300-bus instances with $u=10\%, 20\%$, and all the 1354-bus instances.

Table \ref{tab:exp:DCOPF:exact_verification} reports the solving times of the compact formulation and bilevel formulation on solved instances, and indicate the relative speedup of the compact formulation, defined as
\(\frac{t_{\text{Bilevel}} - t_{\text{Compact}}}{t_{\text{Compact}}} \times 100\%.\)
The compact formulation consistently outperforms the bilevel formulation, except for two small instances which are solved by both formulations in under 2 seconds.
Notably, the compact formulation achieves higher speedups on larger systems, with up to 209.5\% speedup on the 300-bus system with 2\% input perturbation.

Next, Table \ref{tab:exp:DCOPF:exact_verification_hard} reports the performance of two formulations on open instances, where none of the formulations can solve the instances to optimality within 6 hours, primarily due to the weak linear relaxation.
In this case, the comparison focuses on the final primal and dual objectives.
The gains in percentage under the Primal and Dual columns report the performance improvements on the primal and dual side, respectively: \revision{\(\frac{\text{Primal}_{\text{Compact}} - \text{Primal}_{\text{Bilevel}}}{\text{Primal}_{\text{Bilevel}}} \times 100\%\)} and 
\(\frac{\text{Dual}_{\text{Bilevel}} - \text{Dual}_{\text{Compact}}}{\text{Dual}_{\text{Bilevel}}} \times 100\%\).
As reported in Table \ref{tab:exp:DCOPF:exact_verification_hard}, the compact formulation finds substantially better primal solutions than the bilevel formulation, especially for harder instances. This is further supported in Figure \ref{fig:primal_obj_val_along_time:DCOPF}, where the compact formulation consistently converges faster. On the dual side, the compact formulation is slightly better than its bilevel counterpart.
Since the compact formulation systematically outperforms the bilevel formulation, the next experiments focus on the compact formulation.

\paragraph{Effectiveness of PGA-VFA}
Table \ref{tab:exp:DCOPF:pga_vs_compact} focuses on the performance of the proposed PGA-VFA and compact formulation \revision{in identifying feasible solutions} on large power systems (300 bus and 1354 bus).
The performance on small systems is deferred to Appendix \ref{appendix:dcopf:more_results} since the compact formulation solves those instances within 60 seconds. 
PGA-VFA is particularly good at finding high-quality primal solutions within short computing time.
It finds near-optimal solutions on the 300-system orders of magnitude faster than the compact formulation with a reference load vector as the warm-start.
\revision{For the $300$-bus system with 1\% perturbation, PGA-VFA finds a feasible solution with objective $562348.33$ in 7.09s. Meanwhile, Gurobi solves the compact MILP formulation in 166.01s, reporting an optimal value of $564263.39$. It takes Gurobi 114.79s to find a feasible solution at least as good as the one found by PGA-VFA (in 7.09s).}
The compact model cannot find the same quality of primal solutions as PGA-VFA on the 1354-systems within 6 hours.

Finally, Table \ref{tab:pga_impact} reports the impact of warm-starting the compact formulation with PGA-VFA.
The compact model with the PGA-VFA solution as a warm start always outperforms its counterparts with the reference load as the warm start.
Observe also that the optimization barely improves on the PGA-VFA solution.

\section{Optimality Verification for Knapsack}
\label{sec:Knapsack}

One major benefit of the proposed compact formulation is that it may provide quality guarantees for optimization proxies that approximate non-convex parametric optimization problems.
This section illustrates this capability on the knapsack problem.

\subsection{Knapsack Formulation}
Consider a knapsack problem with $K$ items
\begin{align}
\label{eq:knapsack}
    \max_{\vy} \quad
    \left\{
        \vv^T\vy
    \ \middle| \
        \vw^T\vy \leq l, 
        \vy \in \{0,1\}^K
    \right\},
\end{align}
where $\vv$ denotes the value of the items, $l$ is the knapsack capacity and $\vw$ denotes the weight of the items.
The binary decision variable $\vy_k = 1$ indicates putting the item $k$ in the knapsack and vice versa.

\subsection{Optimality Verification}

\begin{figure}[!t]
\centering
\includegraphics[width=0.95\columnwidth]{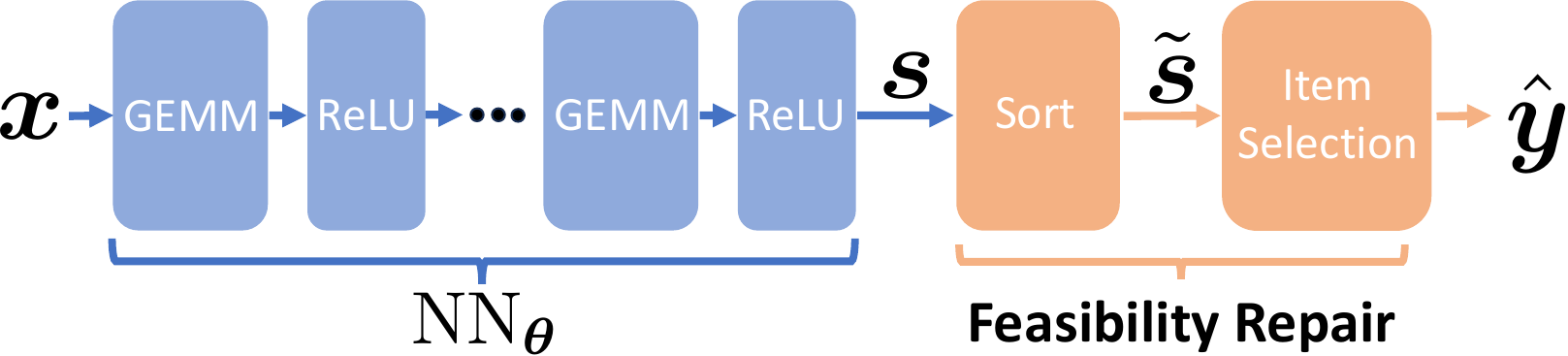}
\caption{An Optimization Proxy for the Knapsack Problem.}
\label{fig:optimization_proxies:knapsack}
\vspace{-1em}
\end{figure}

This paper considers the optimality verification of an optimization proxy for the knapsack problem. 
The proxy is summarized in Figure \ref{fig:optimization_proxies:knapsack}.
\revision{First, a fully-connected neural network with ReLU activation predicts a score for each item, where a higher score indicating higher desirability of the item.
Then, a novel feasibility repair step sorts the items by the predicted score in a descending order and adds items to the knapsack following the order until reaching the knapsack's capacity limit, thus enforcing $\vw^{T}\vy \, {\leq} \, l$.}
{\em Another contribution of this paper is a compact formulation of this repair layer that is presented in Appendix \ref{appendix:knapsack:compact}}.
Because the parametric optimization is non-convex, the bilevel formulation in (\ref{eq:opt_verification:bilevel}) cannot be reformulated into a single level.

\subsection{Numerical Evaluation}
The data for the experiments is generated by perturbing the total capacity of the knapsack and the item values, where the scaler $l$ of the total capacity and the item values $\vv$ are sampled from the uniform distribution $U[80\%, 120\%]$.
More information of the knapsack proxies is detailed in the Appendix \ref{appendix:knapsack}.
The input domain is defined as:
\begin{align*}
    \gX &= \{ \alpha l,\mBeta \vv \ | \ {-}u \leq \alpha-1 \leq u, {-}\mathbf{u} \leq \mBeta-1 \leq \mathbf{u}\},
\end{align*}
where $u$ controls the size of the input domain.


\begin{table}[!t]
    \centering
    \caption{Optimality Verification for Knapsack}
    \label{tab:optimality:knapsack}
    {
    \aboverulesep = 0.3mm \belowrulesep = 0.3mm
    \begin{tabular}{lrrr}
        \toprule[0.15ex]
        \#items & \%u & Gap (\%) & Time (s) \\
        \midrule[0.1ex]
        10 
         & 1 & 0.0 & 0.4 \\
         & 5 & 0.0 & 0.4 \\
         & 10 & 0.0 & 0.4 \\
         & 20 & 0.0 & 0.7 \\
        \midrule[0.1ex]
        50 
         & 1 & 0.0 & 89.1 \\
         & 5 & 0.0 & 117.3 \\
         & 10 & 0.0 & 642.8 \\
         & 20 & 142.0 & t.o. \\
        \midrule[0.1ex]
        80 
         & 1 & 0.0 & 1928.3 \\
         & 5 & 0.0 & 5997.8 \\
         & 10 & 230.0 & t.o. \\
         & 20 & 1350.0 & t.o. \\
        \bottomrule[0.15ex]
    \end{tabular}}
    \\
    \footnotesize{$^{*}$ Solved with 16 threads and time limits of 6 hours.}
    \vspace{-1em}
\end{table}

Table \ref{tab:optimality:knapsack} demonstrates the compact formulation can verify optimization proxies for Knapsack. {\em It highlights a key benefit of the compact formulation: its ability to verify non-convex optimization problems.}

\section{Conclusion}
\label{sec:conclusion}

The paper presents a novel compact formulation for optimality verification of optimization proxies.
It offers substantial computational benefits over the traditional bilevel formulations and can verify non-convex optimization problems.
The paper also introduces a \revision{Projected Gradient Attack} with a value function approximation as a primal heuristic for effectively finding high-quality primal solutions.
The methodology is applied to large-scale DC-OPF and knapsack problems, incorporating new MILP encodings for the feasibility layers.
Extensive experiments demonstrate the efficacy of the methodology in verifying proxies for DC-OPF and knapsack problems, highlighting its computational advantages.
Future works will investigate the scalability of the methodology on large industrial instances with the coupling with spatial branch and bound and $\alpha,\beta$-CROWN, and extend the verification on auto-regression-based optimization proxies.

\paragraph{Limitations}
\revision{The proposed compact formulation is more general than prior state-of-the-art, and has fewer variables and constraints. Nevertheless, this exact verification scheme, which eventually relies on solving MIP problems, shares common limitations with existing exact verification methods, as highlighted in \cite{wang2021betacrown, zhang2022general, ferrari2022complete}.
State-of-the-art exact verification solvers primarily focus on ReLU networks, and do not always support arbitrary linear/discrete/nonlinear constraints which are needed for optimality verification.
Thus, MIP solvers are the only existing tools capable of solving optimality verification problems.
MIP solvers are known to struggle when solving large verification instances, especially for finding high-quality solutions (which correspond to adversarial examples).
Therefore, scalability issues still exist when solving large-scale verification problems, especially with deep neural networks and large input space.
Exact optimality verification on large-scale industrial instances, and exploring its effectiveness for nonlinear activations and non-perceptron architectures, e.g.,Transformers, are promising research directions.}

\section*{Acknowledgements}

This research was partially supported by NSF awards 2007164 and 2112533, and ARPA-E PERFORM award DE-AR0001280.

\section*{Impact Statement} \label{sec:impact_statement}
Optimization proxies have been increasingly used in critical infrastructures such as power systems and supply chains, offering the potential to catalyze \revision{substantial} advancements.
These advancements include facilitating the transition of power systems to high-renewable grids, as well as enhancing the operational resilience and sustainability of supply chain operations. In this context, ensuring the quality and reliability of optimality proxies is essential for their practical deployment.
The compact optimality verification and the gradient-based primal heuristic play a pivotal role in addressing these needs. By offering a more reliable foundation for the deployment of optimization proxies, this research has the potential to substantially impact human life and contribute to social welfare improvement.

\bibliography{main}
\bibliographystyle{icml2024}

\appendix
\onecolumn
\section{More information of DCOPF Optimization Proxies}

\subsection{Detailed Modeling of Compact Formulation} \label{appendix:dcopf:compact}

\begin{figure}[!t]
\centering
\includegraphics[width=0.9\columnwidth]{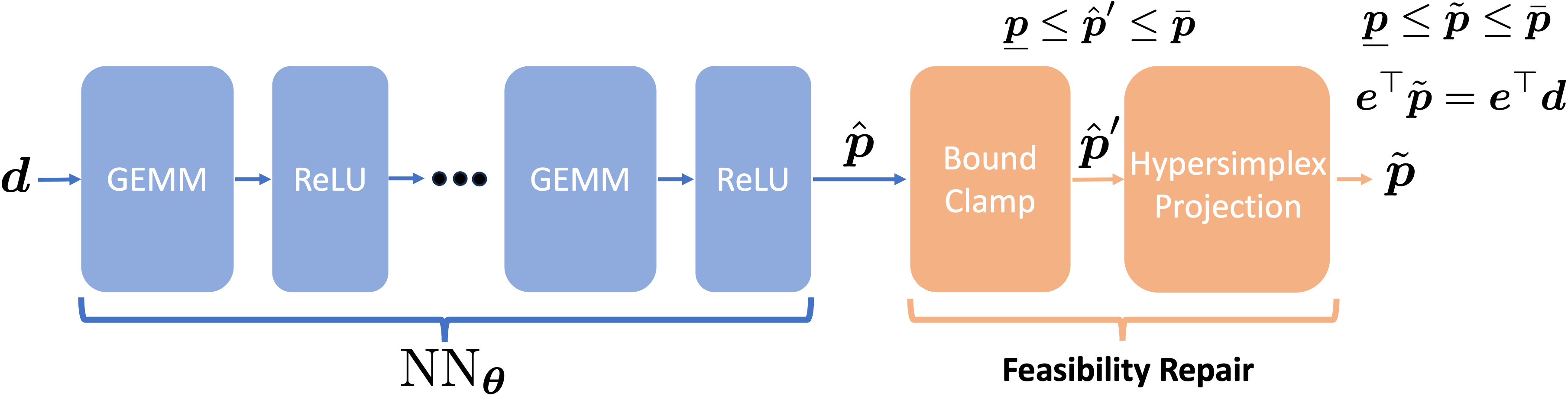}
\caption{Optimization Proxies for DCOPF}
\label{fig:optimization_proxies:dcopf:appendix}
\end{figure}

As shown in Figure \ref{fig:optimization_proxies:dcopf:appendix}, the proposed DC-OPF proxy integrates a multilayer perceptron (MLP), a bound clamping, and a hypersimplex projection layer.
The MLP can be translated into a Mixed Integer Linear Program (MILP) by modeling its general matrix multiplication (GEMM) blocks with linear constraints and its ReLU activations using binary variables and linear constraints, as detailed in prior works \cite{tjeng2017evaluating,bunel2018unified}.

\paragraph{Bound Clamp}
The bound clamping layer is designed to enforce minimum and maximum generation limits \eqref{eq:DCOPF:dispatch_bounds}.
It can be expressed as the composition of two ReLU layers, effectively clamping the input between specified maximum and minimum power generation limits:
\begin{subequations}
    \label{eq:DCOPF:VeriOPT:clamp_milp}
    \begin{align}
    \hat{\vp}' & = \text{clamp}(\hat{\vp}, \pgmin, \pgmax) 
        = \min( \max(\hat{\vp}, \pgmin), \pgmax)\\
    & = -\text{ReLU}(-\text{ReLU}(\hat{\vp} - \pgmin)-\pgmin + \pgmax) + \pgmax.
    \end{align}
\end{subequations}
The clamp operator clips all entries in $\hat{\vp}$ into the range $[\pgmax, \pgmin]$.
It can be formulated as linear constraints and binary decision variables, leveraging the ReLU MILP representation from the aforementioned studies \cite{tjeng2017evaluating,bunel2018unified}.

\paragraph{Hypersimplex Projection}
The Hypersimplex Projection takes as input $\vp \in \gH = \{\pg \in \sR^B | \pgmin \leq \pg \leq \pgmax\},$ and outputs $\tilde{\vp} \in \gS$, where $\gS$ is the hypersimplex
\begin{subequations}
    \begin{align*}
        \gS = \{\pg \in \sR^B | \pgmin \leq \pg \leq \pgmax, \bm{e}^{\top} \pg = \bm{e}^{\top} \pd\}.
    \end{align*}
\end{subequations}
The input vector $\pg$ must satisfy minimum and maximum generation limits \eqref{eq:DCOPF:dispatch_bounds}, and the output vector $\tilde{\vp}$ jointly satisfies minimum/maximum generation limits \eqref{eq:DCOPF:dispatch_bounds} and power balance constraint \eqref{eq:DCOPF:power_balance}.
The projection adjusts all entries in $\pg$ uniformly until either the total power generation matches the total demand, or the entries in $\pg$ reach their bounds.
Formally, $\tilde{\vp}$ is obtained as the unique solution of the system of equations
\begin{subequations} 
\label{eq:DCOPF:VeriOPT:hypersimplex}
\begin{align}
    \tilde{\pg} &= \text{clamp}(\hat{\pg}' + \delta, \pgmin, \pgmax), \label{eq:DCOPF:VeriOPT:proportional_response} \\
    \bm{e}^{\top} \tilde{\pg} &= \bm{e}^{\top} \pd,  \label{eq:DCOPF:VeriOPT:power_balance}
\end{align}
\end{subequations}
where $\delta \in \mathbb{R}$ is a scalar.
Note that \eqref{eq:DCOPF:VeriOPT:hypersimplex} can reduced to a uni-dimensional problem in $\delta$ by substituting out $\tilde{\vp}$.
Namely, letting $f(\delta) = \bm{e}^{\top}\text{clamp}(\hat{\pg}' + \delta, \pgmax, \pgmin)$, \eqref{eq:DCOPF:VeriOPT:power_balance} reduces to $f(\delta) = \bm{e}^{\top} \pd$.
The uniqueness of the solution to \eqref{eq:DCOPF:VeriOPT:hypersimplex} then follows from the fact that $f$ is monotonically increasing.
$\delta$ can be effectively computed using binary search (see Algorithm \ref{alg:hyper_simplex_proj}).
It can be easily parallelized in PyTorch \cite{paszke2019pytorch} with its subgradient computed using auto-differentiation.

\begin{algorithm}
   \caption{Hypersimplex Projection via Binary Search}
   \label{alg:hyper_simplex_proj}
\begin{algorithmic}
    \STATE {\bfseries Input:} Initial dispatch $\tilde{\pg}$, Dispatch bounds $\pgmax$, $\pgmin$, demand $\pd$, numerical tolerance $\epsilon$
    \STATE Initialize $\bar{\delta} = \max(\pgmax - \pgmin)$, $\ubar{\delta} = -\max(\pgmax - \pgmin)$, $\delta = (\ubar{\delta} - \bar{\delta}) / 2$, $D = \bm{e}^{\top} \pd$
    \WHILE{$|\ubar{\delta} - \bar{\delta}| \geq \epsilon$ or $|f(\delta) - D| \geq \epsilon$}
        \IF{$f(\delta) \ge D$}
            \STATE $\bar{\delta} = \delta$
        \ELSE
            \STATE $\ubar{\delta} = \delta$
        \ENDIF
         \STATE $\delta = (\ubar{\delta} - \bar{\delta}) / 2$
    \ENDWHILE
    \\ \textbf{return } $\delta$
\end{algorithmic}
\end{algorithm}

The Hypersimplex projection layer is then represented as an MILP by encoding the system of equations \eqref{eq:DCOPF:VeriOPT:hypersimplex}.
Namely, constraint \eqref{eq:DCOPF:VeriOPT:proportional_response} is encoded as an MILP following Equation \eqref{eq:DCOPF:VeriOPT:clamp_milp}, and constraint \eqref{eq:DCOPF:VeriOPT:power_balance} is linear.
The compact optimality verification problem for DC-OPF proxies thus reads:
\begin{subequations}
\label{eq:DCOPF:Veriopt}
\begin{align}
    \max_{\pd \in \mathcal{X}} \quad & \bm{c}^{\top} \pg + \Mth \bm{e}^{\top}\xith - \bm{c}^{\top} \tilde{\pg} - \Mth \bm{e}^{\top}\tilde{\xi}^{\text{th}}\\
    \text{s.t.} \quad
        & \hat{\pg} = \text{NN}_\vtheta(\pd), 
        \label{eq:DCOPF:VeriOPT:NN} \\
        & \tilde{\vf} = \mH(\tilde{\pg} - \pd), 
        \label{eq:DCOPF:VeriOPT:thermal} \\
        & \tilde{\xi}^{\text{th}} =  \max\{\tilde{\bm{f}} - \pfmax, - \pfmax - \tilde{\bm{f}}, \boldsymbol{0} \}, 
        \label{eq:DCOPF:VeriOPT:thermal_violation} \\
        & (\ref{eq:DCOPF:VeriOPT:clamp_milp}), (\ref{eq:DCOPF:VeriOPT:hypersimplex}),\\
        & (\ref{eq:DCOPF:power_balance}) - (\ref{eq:DCOPF:dispatch_bounds}), \\
        & \vp \in \sR^{B}, \hat{\pg} \in \sR^{B}, \tilde{\pg} \in \sR^{B}, \xith \in \sR_{+}^{E}, \tilde{\xi}^{\text{th}} \in \sR_{+}^{E} \label{eq:DCOPF:Veriopt:variables}
\end{align}
\end{subequations}
Constraint \eqref{eq:DCOPF:VeriOPT:NN} encodes the inference of the neural network, which could be linearized by introducing binary decision variables \citep{tjeng2017evaluating,bunel2018unified}.
Constraint \eqref{eq:DCOPF:VeriOPT:thermal} and \eqref{eq:DCOPF:VeriOPT:thermal_violation} compute the thermal violation of the predicted dispatch.
Before solving formulation \eqref{eq:DCOPF:Veriopt}, all clamp and $\max$ operators are linearized by introducing binary decision variables, which leads to a Mixed Integer Linear Program (MILP).

\subsection{Detailed Modeling of Bilevel Formulation}
\label{appendix:dcopf:bilevel} 

The Bilevel formulation of DC-OPF proxy reads:
\begin{subequations}
\label{eq:DCOPF:bilevel}
\begin{align}
    \max_{\pd \in \mathcal{X}} \quad & \bm{c}^{\top} \pg + \Mth \bm{e}^{\top}\xith - \bm{c}^{\top} \tilde{\pg} - \Mth \bm{e}^{\top}\tilde{\xi}^{\text{th}}\\
    \text{s.t.} \quad
        & (\ref{eq:DCOPF:VeriOPT:NN}) - (\ref{eq:DCOPF:VeriOPT:thermal_violation}), (\ref{eq:DCOPF:Veriopt:variables}) \\
        & (\ref{eq:DCOPF:VeriOPT:clamp_milp}), (\ref{eq:DCOPF:VeriOPT:hypersimplex}),\\
        & \pg, \xith = \text{DC-OPF}(\pd)
\end{align}
\end{subequations}
where lower level $\text{DC-OPF}(\pd)$ outputs the optimal dispatch and thermal violations by solving formulation (\ref{eq:DCOPF}) with input $\pd$.

To enhance the tractability, the work in \cite{nellikkath2021physics} reformulates $\text{DC-OPF}(\pd)$ with KKT conditions:

\begin{subequations} \label{eq:DCOPF:KKT}
\begin{align}
    (\ref{eq:DCOPF:power_balance}) - (\ref{eq:DCOPF:variables}) \\
    \ve \lambda + \mH^{\top}\nuThMin - \mH^{\top} \nuThMax + \muPgMin - \muPgMax & = c \label{eq:DCOPF:KKT:dual}\\
    \nuThMin + \nuThMax + \zeta &= \Mth \ve\\
    \muPgMin, \muPgMax, \nuThMin, \nuThMax, \zeta & \geq 0 \label{eq:DCOPF:KKT:dual:variable_bounds}\\
    \nuThMin^{\top} (\mH \pg + \xith + \pfmax - \mH \pd)
        & = 0 \label{eq:DCOPF:KKT:compl:nu_thm_min}\\
    \nuThMax^{\top} (-\mH \pg + \xith + \pfmax + \mH \pd)
        & = 0 \label{eq:DCOPF:KKT:compl:nu_thm_max}\\
    \muPgMin^{\top} (\pg - \pgmin)
        & = 0 \label{eq:DCOPF:KKT:compl:mu_pg_min}\\
    \muPgMax^{\top} (\pgmax - \pg)
        & = 0 \label{eq:DCOPF:KKT:compl:mu_pg_max}\\
    \zeta^{\top} \xith
        & = 0 \label{eq:DCOPF:KKT:compl:xi_thm}
   \end{align}
\end{subequations}

Constraints (\ref{eq:DCOPF:power_balance}) - (\ref{eq:DCOPF:variables}) model the primal feasibility.
Constraints (\ref{eq:DCOPF:KKT:dual}) - (\ref{eq:DCOPF:KKT:dual:variable_bounds}) model the dual feasibility, where $\muPgMax$ and $\muPgMin$ denote the dual variables for the upper and lower bounds in Constraint \eqref{eq:DCOPF:dispatch_bounds}, respectively.
The $\nuThMin, \nuThMax$ are the dual variables for the upper and lower thermal limits in Constraints \eqref{eq:DCOPF:PTDF:max} and \eqref{eq:DCOPF:PTDF:min}, respectively.
$\lambda$ denotes the dual variable for the power balance constraint \eqref{eq:DCOPF:power_balance}.
$\zeta$ denotes the dual variable for the non-negativity of thermal violations.
Constraints \eqref{eq:DCOPF:KKT:compl:nu_thm_min}-\eqref{eq:DCOPF:KKT:compl:xi_thm} model the complementary slackness, which can be reformulated as an MILP using the standard big-M formulation.
Note that valid bounds on all primal variables can be derived from the primal formulation.
Dual variables $\nuThMin, \nuThMax, \zeta$ are naturally bounded by $\Mth$.
For the remaining dual variables, a large $M$ value is selected.


Finally, the reformulated Bilevel formulation reads
\begin{subequations}
\label{eq:DCOPF:Veriopt:Bilevel}
\begin{align}
    \max_{\pd \in \mathcal{X}} \quad & \bm{c}^{\top} \pg + \Mth \bm{e}^{\top}\xith - \bm{c}^{\top} \tilde{\pg} - \Mth \bm{e}^{\top}\tilde{\xi}^{\text{th}}\\
    \text{s.t.} \quad
        & (\ref{eq:DCOPF:VeriOPT:NN}) - (\ref{eq:DCOPF:VeriOPT:thermal_violation}), (\ref{eq:DCOPF:Veriopt:variables}) \\
        & (\ref{eq:DCOPF:VeriOPT:clamp_milp}), (\ref{eq:DCOPF:VeriOPT:hypersimplex}),\\
        & (\ref{eq:DCOPF:power_balance}) - (\ref{eq:DCOPF:variables}) \\
        & (\ref{eq:DCOPF:KKT:dual}) - (\ref{eq:DCOPF:KKT:compl:xi_thm})
\end{align}
\end{subequations}

\subsection{Optimization-based Bound Tightening} \label{appendix:dcopf:obbt}
The MILP formulations introduce a significant number of binary decision variables. 
The convergence of these MILP problems becomes notably challenging due to the poor quality of their linear relaxations. 
Consequently, tightening the bounds of the input variables for activation functions is crucial to enhancing the relaxations and improving the overall solution process. 
This paper uses Optimization-Based Bound Tightening (OBBT) to refine variable bounds and solution efficiency \cite{caprara2010global,zhao2024bound}.

Consider a neural network with an input vector $ \vx^{(0)} := \vx_0 \in \mathbb{R}^{n_0} $. In this network, $ n_i $ represents the number of neurons in the $ i $-th layer. The network consists of $ L $ layers, where each layer $ i $ has an associated weight matrix $ \mW^{(i)} \in \mathbb{R}^{n_i \times n_{i-1}} $ and a bias vector $ \vb^{(i)} \in \mathbb{R}^{n_i} $, for $ i \in \{1, \ldots, L\} $. Let $ \vy^{(i)} $ denote the pre-activation vector and $ \vx^{(i)} $ the post-activation vector at layer $ i $, with $ \vx^{(i)} = \sigma(\vy^{(i)}) $ where $ \sigma $ could be any activation function. Define $f$ to be the desired property for any input $\vx^{(0)} \in \mathcal{C}$.

The optimization problem is as follows:
\begin{equation}
\begin{aligned}
\min \quad & f(\vy^{(L)}) \\
\text{s.t.} \quad & \vy^{(i)} = \mW^{(i)}\vx^{(i-1)} + \vb^{(i)}, & \forall i \in \{1, \ldots, L\}, \\
& \vx^{(i)} = \sigma(\vy^{(i)}), & \forall i \in \{1, \ldots, L-1\}, \\
& \vx^{(0)} \in \mathcal{C}.
\end{aligned}
\label{eq:model1}
\end{equation}

To tighten variable bounds, OBBT solves two optimization subproblems for each neuron to determine its maximum and minimum bounds \cite{caprara2010global,zhao2024bound}. Specifically, let $\vy^{(t)}_k$ denote the $k$-th neuron at layer $t$ subject to network constraints. Given $0\leq t \leq L$, the problem states:
\begin{equation} \label{eq:obbt_instances}
\begin{aligned}
\text{max/min} \quad & \vy^{(t)}_k \\
\text{s.t.} \quad & \vy^{(i)} = \mW^{(i)}\vx^{(i-1)} + b^{(i)}, & \forall i \in \{1, \ldots, t\}, \\
& \vx^{(i)} = \sigma(\vy^{(i)}), & \forall i \in \{1, \ldots, t-1\}, \\
& \vx^{(0)} \in \mathcal{C}, \\
& \vy^{(i)}_l \leq \vy^{(i)} \leq \vy^{(i)}_u, & \forall i \in \{1, \ldots, t\},\\
& \vx^{(i)}_l \leq \vx^{(i)} \leq \vx^{(i)}_u & \forall i \in \{1, \ldots, t-1\}.
\end{aligned}
\end{equation}

\subsection{More results on optimality verification of DCOPF proxies} \label{appendix:dcopf:more_results}
Table \ref{tab:exp:DCOPF:pga_vs_compact_small} reports the final objective and solving time on small power systems, where Compact formulation can solve all instances to optimality within 1 minute.

\begin{table}[!ht]
    \centering
    \caption{Comparisons of PGA-VFA and Compact formulation on final objective and solving time on small power systems. T2PGA reports the time that the compact formulation takes to reach the same primal objective of the proposed PGA-VFA. Compact formulation is very effective at finding optimal solutions for small systems i.e., it solves all instances to optimality within 1 minute.}
    \label{tab:exp:DCOPF:pga_vs_compact_small}
    \begin{tabular}{lrrrrrrr}
        \toprule
            &
            & \multicolumn{2}{c}{PGA-VFA}
            & \multicolumn{2}{c}{Compact} \\
        \cmidrule(lr){3-4}
        \cmidrule(lr){5-6}
        System &  \%$u$ &  ObjVal (\$) & Time (s) &  ObjVal (\$) & Time (s) & T2PGA (s) \\
        \midrule
        57 
         & 0 & 127.36 & 2.17 & 141.36 & 0.31 & 0.10 \\
         & 1 & 137.35 & 1.44 & 142.29 & 0.59 & 0.11 \\
         & 2 & 143.70 & 4.31 & 151.54 & 0.54 & 0.12 \\
         & 5 & 171.53 & 3.35 & 179.12 & 1.10 & 1.10 \\
         & 10 & 175.89 & 2.28 & 432.20 & 1.21 & 1.10 \\
         & 20 & 1527.72 & 5.59 & 3837.33 & 0.54 & 0.30 \\
        \midrule
        118 
         & 0 & 752.29 & 1.43 & 6216.97 & 1.20 & 0.74 \\
         & 1 & 4615.37 & 6.96 & 7917.56 & 2.07 & 1.17 \\
         & 2 & 4615.37 & 7.61 & 9607.68 & 3.08 & 2.00 \\
         & 5 & 4615.37 & 4.44 & 14329.88 & 12.08 & 3.00 \\
         & 10 & 5448.69 & 6.26 & 68556.32 & 14.14 & 6.67 \\
         & 20 & 6048.77 & 9.04 & 69109.46 & 59.41 & 17.27 \\
         \bottomrule
    \end{tabular}
\end{table}

\section{More information of Knapsack}  \label{appendix:knapsack}

Recall the knapsack problem formulation
\begin{subequations}
\label{eq:knapsack:primal}
\begin{align}
    \max_{\vy} \quad 
        & \vv^T\vy\\
    s.t. \quad
    & \vw^T\vy \leq l, \label{eq:knapsack:capacity}\\
    & \vy \in \{0,1\}^K.\label{eq:knapsack:binary}
\end{align}
\end{subequations}

\subsection{Detailed Modeling of Compact Formulation} \label{appendix:knapsack:compact}
\begin{figure}[!t]
\centering
\includegraphics[width=0.9\columnwidth]{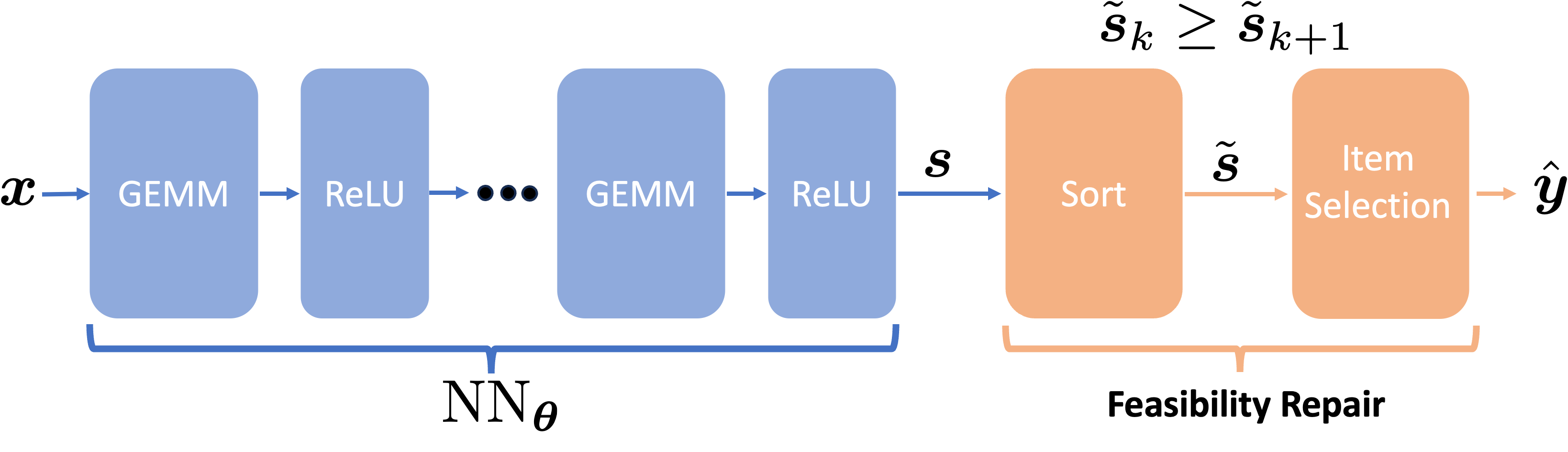}
\caption{Optimization Proxies for Knapsack}
\label{fig:optimization_proxies:knapsack:appendix}
\end{figure}

Figure \ref{fig:optimization_proxies:knapsack:appendix} illustrates the Knapsack proxies. First, a fully-connected neural network with ReLU activation predicts a score for each item, where a higher score indicating higher desirability of the item. Then, a feasibility repair step sorts the items by the predicted score in a descending order and adds items to the knapsack following the order until reaching the knapsack's capacity limit, addressing constraint \eqref{eq:knapsack:capacity}.
The MLP is trained disjointly with the feasibility repair steps in the training due to their non-differentiability.
The neural network is trained using supervised learning to minimize the distances of the predicted scores to the ground truth:
$l(\vs, \vs^*)  = \|\vs - \vs^*\|_2$.
The ground truth score is computed heuristically:
$\vs^* = \frac{\vv}{\vw}$, where $\vv$ denotes the value of the items and $\w$ denotes the weight of the items. Note that the score is exactly the solution of the continuous relaxation of Knapsack problems.

The following part focuses on formulating the sorting and item selection operators into MILP.
\paragraph{Sorting}
Denote the predicted score $\vs$ from the neural network, the sorting operation sorts the item id with the score descendingly.
The sorting operation is modeled by introducing a permutation matrix:
\begin{subequations}
\label{eq:knapsack:optimality_verification:sorting}
\begin{align}
    & \tilde{\vs} = \vp\vs, \label{eq:knapsack:Veriopt:permutation_s}\\
    & \tilde{\vs}_{k} \ge \tilde{\vs}_{k+1}, \forall k \in [K] \label{eq:knapsack:Veriopt:sort_descending}\\
    & \sum_i^K \vp_{i,k} = 1, \forall k \in [K], \label{eq:knapsack:Veriopt:permutation1}\\
    & \sum_k^K \vp_{i,k} = 1, \forall i \in [K], \label{eq:knapsack:Veriopt:permutation2}\\
    & \vp \in \{0,1\}^{K\times K}, \label{eq:knapsack:Veriopt:permutation3}\\
    & \vs \in \mathbb{R}_+^{K}, \tilde{\vs} \in \mathbb{R}_+^{K}.
\end{align}
\end{subequations}
The $\tilde{\vs}$ denotes the score after the sortation.
The $\vp$ denotes the permutation matrix to convert predicted score $\vs$ to sorted score $\tilde{\vs}$.
Constraint (\ref{eq:knapsack:Veriopt:permutation1})-(\ref{eq:knapsack:Veriopt:permutation3}) define the permutation matrix and constraint (\ref{eq:knapsack:Veriopt:permutation_s}) and (\ref{eq:knapsack:Veriopt:sort_descending}) ensure the perturbation matrix encodes the orders of sorting the score descendingly.

\paragraph{Item Selection}
After obtaining the permutation matrix, the item selection operator could be modeled as:

\begin{subequations}
\label{eq:knapsack:optimality_verification:item_selection}
\begin{align}
    & \tilde{\vy} = \vp\hat{\vy}, \label{eq:knapsack:Veriopt:permutation_y} \\
    & \tilde{\vw} = \vp\vw, \label{eq:knapsack:Veriopt:permutation_w} \\
    & \vw^T\hat{\vy} \leq l, \label{eq:knapsack:Veriopt:add_item3}\\
    & \tilde{\vy}_{k} \ge \tilde{\vy}_{k+1}, \forall k \in [K] \label{eq:knapsack:Veriopt:add_item1}\\
    & \sum_{i=1}^k \tilde{\vw}_{i} \geq (1-\tilde{\vy}_k)(l+1), \forall k \in [K] \label{eq:knapsack:Veriopt:add_item2}\\
    & \tilde{\vy} \in \{0,1\}^K, \hat{\vy} \in \{0,1\}^K.
\end{align}
\end{subequations}
Constraint (\ref{eq:knapsack:Veriopt:permutation_y}) and (\ref{eq:knapsack:Veriopt:permutation_w}) permutate the item action and weights.
Constraint (\ref{eq:knapsack:Veriopt:add_item3}) ensures the total weight of the added items is less than the knapsack capacity.
Constraint (\ref{eq:knapsack:Veriopt:add_item1}) ensures items are added sequentially following the sorted scores 
Constraint (\ref{eq:knapsack:Veriopt:add_item2}) ensures that the addition of items continues only if adding another item does not exceed the knapsack's capacity.
Finally, following the reformulation of the operations and Formulation \eqref{eq:optimality_verification_compact}, the compact optimality verification for Knapsack proxies could be written as a MILP and readily be solved by solvers like Gurobi \cite{gurobi}.

Finally, the optimality verification formulation reads:
\begin{subequations}
\label{eq:knapsack:optimality_verification}
\begin{align}
    \max_{\vx = \{\vv, l\} \in \mathcal{X}} \quad 
        & \vv^T\vy - \vv^T\hat{\vy} \label{eq:knapsack:Veriopt:obj}\\
    s.t. \quad
    & \vs = \text{NN}_{\boldsymbol{\theta}}(\vx), \label{eq:knapsack:Veriopt:NN} \\
    & (\ref{eq:knapsack:optimality_verification:sorting}), (\ref{eq:knapsack:optimality_verification:item_selection}), \label{eq:knapsack:Veriopt:feasibility} \\
    & (\ref{eq:knapsack:capacity}), (\ref{eq:knapsack:binary})
\end{align}
\end{subequations}
The objective function \eqref{eq:knapsack:Veriopt:obj}, constraints \eqref{eq:knapsack:Veriopt:NN} and \eqref{eq:knapsack:Veriopt:feasibility} are linearized by introducing binary decision variables, which gives a MILP.

\end{document}